\documentclass[english]{extarticle}
\usepackage[T1]{fontenc}
\usepackage[latin9]{inputenc}
\usepackage{geometry}
\usepackage{color}
\usepackage{babel}
\usepackage{float}
\usepackage{mathrsfs}
\usepackage{amsmath}
\usepackage{amsthm}
\usepackage{amssymb}
\usepackage{graphicx}
\usepackage[numbers]{natbib}
\usepackage[unicode=true,pdfusetitle,
 bookmarks=true,bookmarksnumbered=false,bookmarksopen=false,
 breaklinks=false,pdfborder={0 0 1},backref=page,colorlinks=true]
 {hyperref}
\hypersetup{
 linkcolor=blue, urlcolor=marineblue, citecolor=blue, pdfstartview={FitH}, hyperfootnotes=false, unicode=true}

\makeatletter
\numberwithin{equation}{section}
\numberwithin{figure}{section}
\theoremstyle{plain}
\newtheorem{thm}{\protect\theoremname}
\theoremstyle{plain}
\newtheorem{prop}[thm]{\protect\propositionname}
\theoremstyle{plain}
\newtheorem{lem}[thm]{\protect\lemmaname}
\newcommand{\RR}{\mathbb{R}}
\newcommand{\CC}{\mathbb{C}}
\newcommand{\si}{\sigma}
\newcommand{\D}{\Delta}
\newcommand{\dd}{\delta}

\let\myFoot\footnote
\renewcommand{\footnote}[1]{\myFoot{#1\vspace{3mm}}}
\usepackage{lmodern}
\usepackage[T1]{fontenc}
\usepackage{concmath}
\usepackage{ae,aecompl}
\usepackage[T1]{fontenc}
\usepackage{amsmath}
\usepackage{graphicx}

\usepackage{tikz}

\makeatother

\providecommand{\lemmaname}{Lemma}
\providecommand{\propositionname}{Proposition}
\providecommand{\theoremname}{Theorem}

\begin{document}
\title{On the finiteness of moments of the exit time of planar Brownian motion from comb domains}

\author{~Maher Boudabra and Greg Markowsky\\
{\small {\tt maher.boudabra@monash.edu} ~~ {\tt greg.markowsky@monash.edu }}\\
{\small Department of Mathematics,  Monash University, Australia}}

\maketitle

\begin{abstract}
A comb domain is defined to be the entire complex plain with a collection of vertical slits, symmetric over the real axis, removed. In this paper, we consider the question of determining whether the exit time of planar Brownian motion from such a domain has finite $p$-th moment. This question has been addressed before in relation to starlike domains, but these previous results do not apply to comb domains. Our main result is a sufficient condition on the location of the slits which ensures that the $p$-th moment of the exit time is finite. Several auxiliary results are also presented, including a construction of a comb domain whose exit time has infinite $p$-th moment for all $p \geq 1/2$. 
\end{abstract}



\section{Introduction and statement of main result}
Let $(x_{n})_{n\in\mathbb{Z}}$ be an increasing sequence
of distinct real numbers without accumulation point in $\RR$, let $(b_{n})_{n\in\mathbb{Z}}$ be an associated sequence of positive numbers, and let $\mathscr{M}_{x}$ be the domain 
\[
\mathscr{M}_{x}:=\mathbb{C}\setminus\bigcup_{n\in\mathbb{Z}}I_{n}
\]
where $I_{n}:=\{x_{n}\}\times([b_n,+\infty)\cup(-\infty,-b_n])$.
We shall informally refer to $\mathscr{M}_{x}$ as a {\it comb domain}. 
\begin{figure}[H]
\centering{}\includegraphics[width=8cm,height=8cm,keepaspectratio]{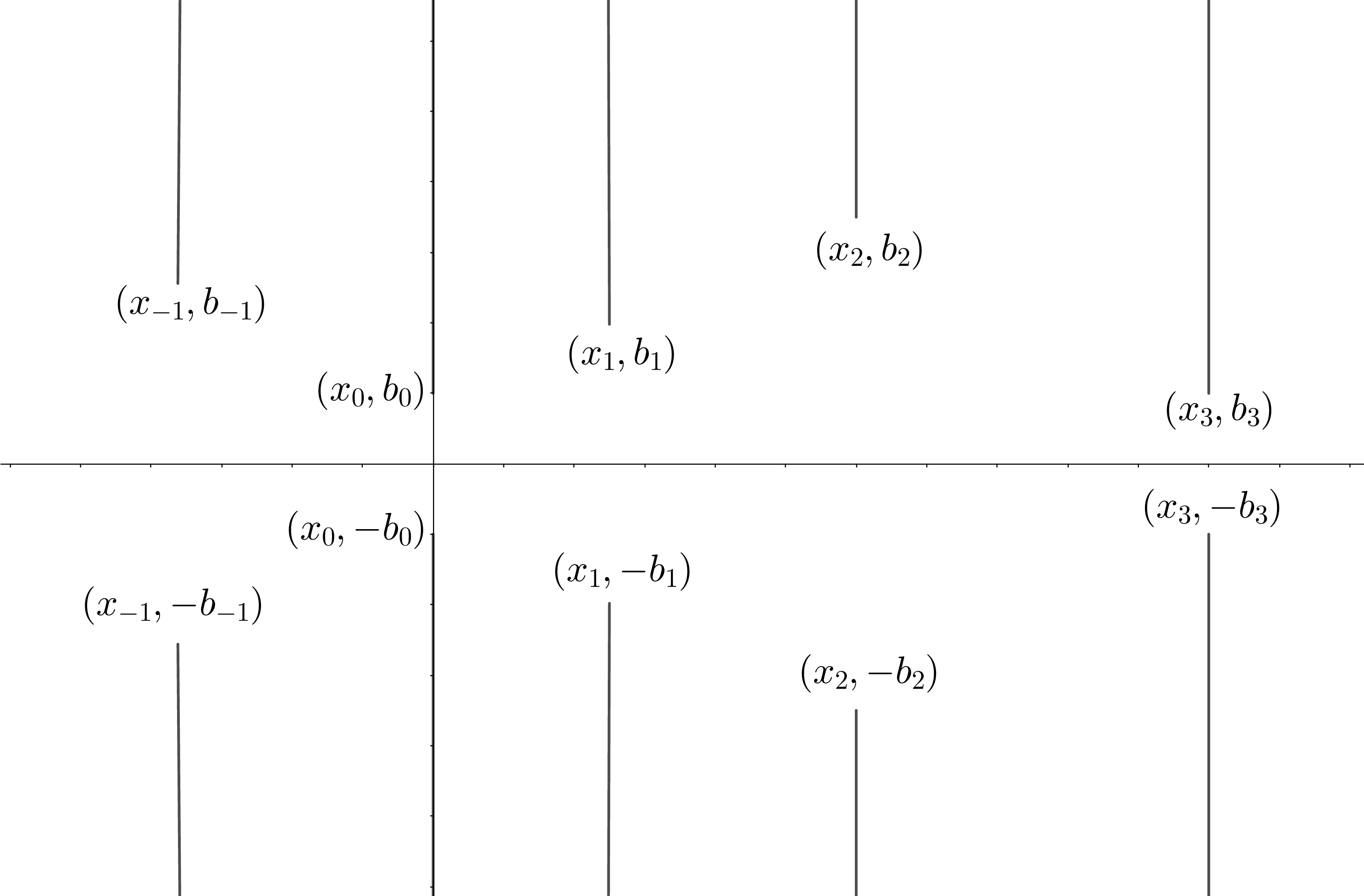}\caption{Illustration of a comb domain.}
\end{figure}
 We consider a planar Brownian motion $Z_{t}$ and denote by $\tau_\Omega$
its exit time from a given domain $\Omega$. The question we will investigate in this paper is to find conditions on the sequences $(x_{n})_{n\in\mathbb{Z}}$ and $(b_{n})_{n\in\mathbb{Z}}$ which imply that $E(\tau_{\mathscr{M}_{x}}^p)<\infty$ for a given $p \in (0,\infty)$. We will derive a sufficient condition for the moment to be finite, but before stating our results, we discuss a bit of motivation for this question.

The moments of $\tau_{\Omega}$ have special importance in two dimensions, as they carry a great deal of analytic and geometric information about the domain $\Omega$. The first major work in this direction seems to have been by Burkholder in \citep{burkholder1977exit}, where it was proved among other things that finiteness of the $p$-th Hardy norm of $\Omega$ is equivalent to finiteness of the $\frac{p}{2}$-th moment of $\tau_{\Omega}$. To be precise, for any simply connected domain $\Omega$
let

$$
{\rm H}(\Omega)=\sup \{p > 0: E((\tau_{\Omega})^p) < \infty\};
$$

note that ${\rm H}(\Omega)$ is proved in \cite[p. 183]{burkholder1977exit} to be exactly
equal to half of the Hardy number of $\Omega$, as defined in \citep{hansen}, which is

$$
{\widetilde{\rm H}}(\Omega)=\sup \{q > 0: \lim_{r \nearrow 1} \int_{0}^{2\pi} |f(re^{i\theta})|^qd \theta < \infty\},
$$

where $f$ is a conformal map from the unit disk onto $\Omega$. This equivalence was used in \cite[p. 183]{burkholder1977exit} to show for instance that $H(W_\alpha) = \frac{\pi}{2 \alpha}$, where $W_\alpha = \{0<Arg(z)<\alpha\}$ is an infinite angular wedge with angle $\alpha$. In fact, coupled with the purely analytic result \cite[Thm 4.1]{hansen} this can be used to determine ${\rm H}(\Omega)$ for any starlike domain $\Omega$ in terms of the aperture of $\Omega$ at $\infty$, which is defined to be the limit as $r \to \infty$ of the quantity
$\alpha_{r,\Omega} = \max \{m(E): E \mbox{ is a subarc of } \Omega \cap \{|z|=r\}\}$; it is not hard to see that this limit always exists for starlike domains. \citep{markowsky} contains a detailed discussion of this, as well as a version of the Phragm\'en-Lindel\"of principle that makes use of the quantity ${\rm H}(\Omega)$. Furthermore, the quantity $E((\tau_{\Omega})^p)$ provides us with an estimate for the tail probability $P(\tau_{\Omega} > \delta)$: by Markov's inequality, $P(\tau_{\Omega} > \delta) \leq \frac{ E((\tau_{\Omega})^p)}{\delta^p}$.
We should also mention that the case $p=1$ is naturally of special interest, and has produced a literature too large to describe here; the case of general $p$ has attracted somewhat less interest, nevertheless the reader interested in other results relating the $p$-th moments of Brownian exit time with the geometry of domains is referred to \cite{banuelos2020bounds, maxpmoment,burchard2001comparison,davis1994moments,dryden2017exit,  helmes1999computing, helmes2001computing,  hurtado2012comparison,hurtado2016estimates,kim2020quantitative, kinateder1999variational, kinateder1998exit, li2003first,mcdonald2013exit, mendez}.


On the other hand, comb domains and their analogues have appeared in a number of recent papers on various topics. One of the most striking instances is in the recent work \citep{gross2019conformal} by Gross, in which the following question was posed and answered: given a measure $\mu$ on $\RR$ with finite second moment, find a simply connected domain $U$ in $\CC$ such that the real part of the random variable $Z_{\tau_U}$ has the distribution $\mu$. If $\mu$ is a discrete distribution, then Gross' construction yields a comb domain. Other examples include \citep{markowsky2018remark}, in which a similar domain was used in order to construct a stopping time related to the winding of Brownian motion, and \citep{karafyllia2019property,karafyllia2019relation, karafyllia2019hardy}, in which similar domains were used as counterexamples to several conjectures concerning harmonic measure posed in \cite{betsakos1998harmonic,betsakos2001geometric}. Note that in Gross' paper in particular (see also \citep{boudabra2020new,boudabra2019remarks,mariano2020conformal}) the moments of the exit time are of importance, and yet it is not simple to show that they are finite for a given comb-like domain. 

Comb domains are never starlike, and therefore the previously established results on ${\rm H}(\Omega)$ do not apply to them. It is also not hard to see that the aperture at $\infty$ of comb domains need not exist. We have therefore needed to devise new methods in order to address this question.

Before discussing our results, let us make a few observations to clarify the problem. To begin with, it may seem that the starting point of the Brownian motion affects whether $E(\tau^p)$ is finite, however this is not so, as shown in \citep{burkholder1977exit}:

\begin{prop} \citep[p.13 (3.13)]{burkholder1977exit}
\label{prop:-For-,} If $U$ is a domain and $E_{a}(\tau_U)<\infty$ for some $a \in U$,
then $E_{w}(\tau_U)<\infty$ for any $w \in U$.
\end{prop}

We may therefore make statements like "$E(\tau_{\Omega}^p)< \infty$" or "$E(\tau_\Omega^p)= \infty$" without specifying a starting point. We next note that exit times are monotonic with respect to domains, as the following proposition shows.

\begin{prop} \label{monotone}

\begin{itemize}
    \item  If $\Omega_1 \subset \Omega_2$ then 
\[
E(\tau_{\Omega_2}^{p})<\infty\,\Longrightarrow\,E(\tau_{\Omega_1}^{p})<\infty.
\]

    \item If $\Omega_n$ is an increasing sequence of domains (i.e. $\Omega_n \subseteq \Omega_{n+1}$) and $\Omega = \cup_{n=1}^\infty \Omega_n$, then $E(\tau_{\Omega_n}^p) \nearrow E(\tau_{\Omega}^p)$.
\end{itemize}

\end{prop}

The proof of the first statement is trivial, and the second is a simple consequence of the monotone convergence theorem. This proposition allows us to clarify the problem a bit. Any comb domain is contained in a translation and dilation of the domain $U=\mathbb{C} \setminus \{0\}\times([1,+\infty)\cup(-\infty,-1])$, and, as the conformal map from the unit disk onto this domain is readily computed, a straightforward application of the aforementioned results by Burkholder in \cite{burkholder1977exit} shows that $E(\tau_U^p) < \infty$ if, and only if, $p<1/2$. It follows that $E(\tau_{\mathscr{M}_{x}}^p) < \infty$ for any comb domain $\mathscr{M}_{x}$, if $p<1/2$. The question is thus only interesting for $p \geq 1/2$, and we will concentrate on these values of $p$ in what follows. The following proposition, which in essence shows that the question we are addressing is reasonable, is proved in Section \ref{proofs}.

\begin{prop} \label{existinf}
For any $p \geq 1/2$, there is a comb domain $\mathscr{M}_{x}$ for which $E(\tau_{\mathscr{M}_{x}}^p)=\infty$.
\end{prop}

In order to state our main result, let us employ the notation $a_n = x_n-x_{n-1}$. Then we have the following.

\begin{thm} \label{bigguy}
Suppose $(x_{n})_{n\in\mathbb{Z}}$ is an increasing sequence (with $x_0 = 0$), and $(b_{n})_{n\in\mathbb{Z}}$ is an associated sequence of positive numbers, such that $\ell=\sup_n \Big(\frac{\max(b_{n-1},b_{n+1})}{\min(a_n,a_{n+1})}\Big) < \infty$. Then there is a number $\theta_0<1$, depending on $\ell$, such that, for any $p>0$, if 

\begin{equation} \label{keyeq}
    \sum_{j=1}^\infty (\max_{|n| \leq j} a_n^2) \theta_0^{j/p} < \infty,
\end{equation} 

then $E(\tau_{\mathscr{M}_{x}}^p)<\infty$.
\end{thm}

We note that this theorem can also be applied in many cases where $\ell=\infty$, since removing slits in the complement of the domain can only increase the moments of the exit time. Therefore, if a collection of slits can be removed from the complement of $\mathscr{M}_{x}$ such that $\ell$ becomes finite (if for instance $\ell$ was infinite due to the $a_n$'s being small rather than the $b_n$'s being large) but (\ref{keyeq}) persists then the conclusion of the theorem still holds. As an immediate corollary of the theorem, if $a_n$ is uniformly bounded, or even bounded by any polynomial in $n$, and $b_n$ is uniformly bounded as well, then $\sum_{j=1}^\infty (\max_{|n| \leq j} a_n^2) \theta^j < \infty$ for any  $\theta<1$, and therefore all moments of $\tau_{\mathscr{M}_{x}}$ are finite. We will see later that this theorem can even be extended a bit in order to handle certain sequences where the $a_n$ grow faster than this, for instance certain sequences with exponential growth.



We will prove Proposition \ref{existinf} and Theorem \ref{bigguy} in the next section, and add some concluding remarks in the final section.

\section{Proofs} \label{proofs}

{\it Proof of Proposition \ref{existinf}}

By the monotonicity of moments, it is enough to consider $p= \frac{1}{2}$. Our domain will have $b_n = 1$ for all $n$. Let us first consider a comb domain derived from a finite sequence, that is 

\[
\mathscr{M}_{x}:=\mathbb{C}\setminus\bigcup_{n\in\mathbb\{1, \ldots N\}}I_{n}
\]
where again $I_{n}:=\{x_{n}\}\times([1,+\infty)\cup(-\infty,-1])$.
In this case $\mathscr{M}_{x}$ contains the half plane $\{\Re(z) > x_N\}$. The exit time of a half plane has infinite $\frac{1}{2}$ moment, as discussed in the previous section. By Proposition \ref{monotone}, $E((\tau_{\mathscr{M}_{x}})^{1/2}] = \infty$.
We are now ready to construct an infinite unbounded sequence $(x_{n})_{n\in\mathbb{Z}}$ whose corresponding $\mathscr{M}_{x}$ has infinite $\frac{1}{2}$ moment; in fact, it will even be subdomain of $\{\Re(z)>0\}$ with a one-sided sequence of vertical slits removed, and naturally it can be extended arbitrarily to a two sided sequence if desired. For $c<d$ let $S_{c,d}$ denote the infinite vertical strip $\{b<\Re(z)<c\}$. We will start our Brownian motion at the point 1. Let $x_{1}>1$ be a real number such that $E_1(\tau_{S_{0,x_{1}}}^{1/2})>1$,
which does exist because $E_1(\tau_{S_{0,x}}^{1/2})\nearrow+\infty$ as
$x \nearrow +\infty$ by Proposition \ref{monotone}. Next, consider the domain $U_{2}=S_{0,x_{2}}\setminus I_{1}$
with $I_{1}:=\{x_{1}\}\times([1,+\infty)\cup(-\infty,-1])$, where
$x_{2}$ is chosen so that $E_1(\tau_{U_{2}}^{1/2})>2$, and again this is possible since $\lim_{x \nearrow \infty} E_1((\tau_{S_{0,x}\setminus I_{1}})^{1/2}) = \infty$. Continuing inductively in this way we
construct $U_{n+1}$ from $U_{n}$ by
\[
U_{n+1}=(U_{n}\cap S_{x_{n+1}})\setminus I_{n}
\]
where $I_{n}:=\{x_{n}\}\times([1,+\infty)\cup(-\infty,-1])$, $x_{n}<x_{n+1}$
and $E_1(\tau_{U_{n+1}}^{1/2})>n+1$. The domain 
\[
U:=\bigcup_{n=1}^{\infty}U_{n}=U_{\infty}
\]
 (with $U_{1}:=S_{0,x_{1}}$) is a comb domain that fits the requirement
since 
\[
n\leq E_1(\tau_{U_{n}}^{1/2})\leq E_1(\tau_{U}^{1/2})
\]
for all $n$. Consequently $E_(\tau_{U}^{1/2})=+\infty$, and thus all
moments $E(\tau_{U}^{p})$ are infinite for any $p\in[1/2,+\infty)$. 

\qed

{\it Proof of Theorem \ref{bigguy}}

Before tackling the proof we give some notations and definitions
that we will use. As before, for $c<d$ let $S_{c,d} = \{c<\Re(z)<d\}$.
It will be convenient to think of the sequence $(x_{n})_{n\in\mathbb{Z}}$ as a map from $\mathbb{Z}$ into $\RR$ defined by $x(n) = x_n$, and with inverse $x^{-1}$. Denote the image of this map by ${\cal X} = \cup_{n=-\infty}^\infty \{x_n\}$. We will assume that our Brownian motion starts at $x_0 = 0$. Consider the following sequence
of associated stopping times
\[
\widehat{\tau}_{j}:=\begin{cases}
0 & {\scriptstyle \left(j=0\right)}\\
\inf\{t>\widehat{\tau}_{j-1}\mid R_{t}\in{\cal X}\setminus\{R_{\widehat{\tau}_{j-1}}\}\} & {\scriptstyle \left(j>0\right)}
\end{cases}.
\]
with $R_{t}=\Re (Z_{t})$. More precisely, $\widehat{\tau}_{j}$ encodes
the time of the $j^{th}$ passage of $R_{t}$ at the lines carrying
the slits under the constraint that it is different from the $(j-1)^{st}$
one. Equivalently, $\widehat{\tau}_{j}$ is the first exit time of $Z_{t}$
from $S_{x(x^{-1}(R_{\widehat{\tau}_{j-1}})-1),x(x^{-1}(R_{\widehat{\tau}_{j-1}})+1)}$ after $\widehat{\tau}_{j-1}$.
Finally, let $\tau$ be the exit time from the comb domain and set $\tau_{j}=\tau\wedge\widehat{\tau}_{j}$.
$\tau$ can be expressed as 
\[
\tau=\sum_{j=0}^{\infty}(\tau_{j+1}-\tau_{j}),
\]
whence
\[
E(\tau^{p})^{1/p}\leq\sum_{j=0}^{\infty}E((\tau_{j+1}-\tau_{j})^{p})^{1/p}
\]
thanks to the H\"older-Minkowsky inequality. We need therefore only show that this sum is finite. Note that $\tau_j = \tau_{j+1}$ on the event $\{\tau \leq \tau_j\}$, while $\tau_j$ and $\tau_{j+1}$ are simply equal to $\widehat \tau_j$ and $\widehat \tau_{j+1}$ on $\{\tau_j < \tau\}$. It therefore follows that

\begin{equation} \label{yeungeun}
\begin{split}
E((\tau_{j+1}-\tau_{j})^{p}) & = E((\widehat \tau_{j+1}-\widehat \tau_{j})^{p} 1_{\{\tau_j < \tau\}}) \\
&= E((\widehat \tau_{j+1}-\widehat \tau_{j})^{p} 1_{\{\tau_j < \tau\}} \sum_{n=-j}^j 1_{\{R_{\widehat{\tau}_{j}} = x_n\}}) \\
&= \sum_{n=-j}^j E((\widehat \tau_{j+1}-\widehat \tau_{j})^{p} 1_{\{\tau_j < \tau\}}  1_{\{R_{\widehat{\tau}_{j}} = x_n\}}) \\
& = \sum_{n=-j}^j P(\tau_j < \tau, R_{\widehat{\tau}_{j}} = x_n) E((\widehat \tau_{j+1}-\widehat \tau_{j})^{p}|\tau_j < \tau, R_{\widehat{\tau}_{j}} = x_n) \\
& = \sum_{n=-j}^j P(\tau_j < \tau, R_{\widehat{\tau}_{j}} = x_n) E_{x_n}((\tau_{S_{x(x^{-1}(R_{\widehat{\tau}_{j-1}})-1),x(x^{-1}(R_{\widehat{\tau}_{j-1}})+1)}})^p) \\
& = \sum_{n=-j}^j P(\tau_j < \tau, R_{\widehat{\tau}_{j}} = x_n) E_{0}((\tau_{S_{-a_n,a_{n+1}}})^p) \\
& \leq P(\tau_j < \tau) \max_{|n| \leq j}E_{0}((\tau_{S_{-a_n,a_{n+1}}})^p).
\end{split}
\end{equation}

Note that we have used the strong Markov property in the second-to-last equality, and also that our sum needed only to be over the set $\{|n|\leq j\}$ rather than all of $\mathbb{Z}$ because $R_{\widehat{\tau}_{j}}$ cannot be equal to $x_n$ with $|n|>j$ since $R_{\widehat{\tau}_{0}}=x_0$. In order to estimate this quantity, we need the following lemmas.

\begin{lem} \label{pmoment_strip}
$E_{0}((\tau_{S_{-a_n,a_{n+1}}})^p) \leq \max (a_n, a_{n+1})^{2p}E_{0}((\tau_{S_{-1,1}})^p)$.
\end{lem}
\begin{proof}
For the sake of brevity we denote by $d_n=\max (a_n,a_{n+1})$. A monotonicity argument yields $ E_{0}((\tau_{S_{-a_n,a_{n+1}}})^p) \leq E_{0}((\tau_{S_{-d_n,d_n}})^p) $.  The exit time from the strip $S_{-d_n,d_n}$ is simply the exit time of a one dimensional Brownian motion from the interval $(-d_n,d_n)$ . Hence by scaling, we get $E_{0}((\tau_{S_{-d_n,d_n}})^p) =d_n^{2p}E_{0}((\tau_{S_{-1,1}})^p)$  which completes the proof.
\end{proof}

For the next lemma, let $K^b_{c,d}$ denote the rectangle $S_{c,d} \cap \{-b<\Im(z)<b\}$, and let $I_t = \Im(Z_t)$. Recall also that $\ell=\sup_n \Big(\frac{\max(b_{n-1},b_{n+1})}{\min(a_n,a_{n+1})}\Big) < \infty$.

\begin{lem} \label{probexitstrip}
We have $P(\tau_{j}<\tau)\leq\theta_0^{j}$, where $\theta_0:= 1-\frac{1}{2} P_{0}( |I_{\tau_{K^\ell_{-1,1}}}|=\ell)$.
\end{lem}

\begin{proof}
The proof is by induction. Assume that the statement holds for $j-1$, so that
$
P(\tau_{j}<\tau) = P(\tau_{j-1}<\tau) P(\tau_{j}<\tau|\tau_{j-1}<\tau) \leq \theta_0^{j-1}P(\tau_{j}<\tau|\tau_{j-1}<\tau).
$

Now, if $\tau_{j-1}<\tau$ then $Z_{\tau_{j-1}} \in \mathscr{M}_{x}$ (rather than in its complement). We need to show that, under this assumption, the probability that $Z_{\tau_{j}} \in \mathscr{M}_{x}$ is bounded above by $\theta_0$. This will follow from the strong Markov property if we can show that

\begin{equation} \label{xz}
    1-\sup_{x_n \in {\cal X}, y \in (-b_n,b_n)} P_{x_n + yi} (|I_{\tau_{S_{x_{n-1},x_{n+1}}}}| < \beta_n) = \inf_{x_n \in {\cal X}, y \in (-b_n,b_n)} P_{x_n + yi} (|I_{\tau_{S_{x_{n-1},x_{n+1}}}}| \geq \beta_n) \geq \frac{1}{2} P_{0}( |I_{\tau_{K^\ell_{-1,1}}}| = \ell ),
\end{equation}

where $\beta_n = \max(b_{n-1},b_{n+1})$; note that we are using the fact that on the event $\{R_{\tau_{j-1}}=x_n\}$ the event $\{|I_{\tau_{j}}| \geq \beta_n\}$ is contained in the event $\{\tau_j=\tau\}$. The proof of this depends on two claims.

{\bf Claim 1:} For fixed $x_n \in {\cal X}$, 

$$
\inf_{y \in (-b_n,b_n)} P_{x_n + yi} (|I_{\tau_{S_{x_{n-1},x_{n+1}}}}| \geq \beta_n) = P_{x_n} (|I_{\tau_{S_{x_{n-1},x_{n+1}}}}| \geq \beta_n).
$$

That is, the probability is minimized when $y=0$. To prove this, we employ a coupling argument. Fiz $y>0$, and let $Z_0=x_n$ a.s. Let $\si(z) = \bar z + yi$; note that $\si(z)$ is the reflection over the horizontal line $\D=\{\Im(z) = \frac{y}{2}\}$. Let $H_\D$ be first time that $Z_t$ hits $\D$, and form the process $\widetilde Z_t$ by the rule

$$ \widetilde Z_t = \left \{ \begin{array}{ll}
\si(Z_t) & \qquad  \mbox{if } t < H_\D  \\
Z_t & \qquad \mbox{if } t \geq H_\D\;.
\end{array} \right. $$

\begin{figure}[H]
\centering{}\includegraphics[width=8cm,height=8cm,keepaspectratio]{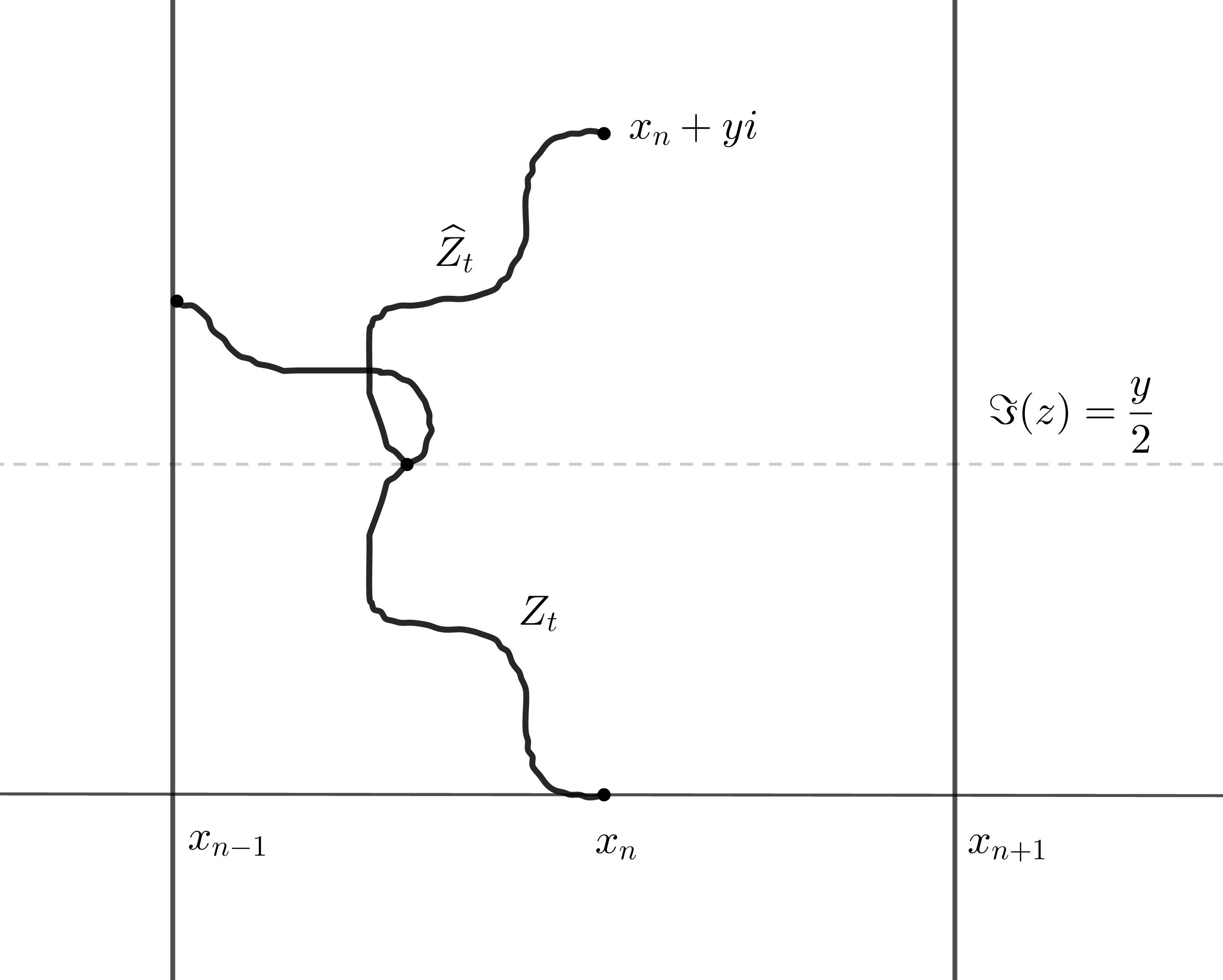}\caption{$Z$ and $\widetilde Z$ coalesce upon hitting $\{\Im(z)=\frac{y}{2}\}$}
\end{figure}

It follows from the Strong Markov property and the reflection invariance of Brownian motion that $\widetilde Z_t$ is a Brownian motion. Let $\widetilde \tau_{S_{x_{n-1},x_{n+1}}}$ denote the first time that $\widetilde Z_t$ exits $S_{x_{n-1},x_{n+1}}$. By the translation invariance of $S_{x_{n-1},x_{n+1}}$ we have $\tau_{S_{x_{n-1},x_{n+1}}} = \widetilde \tau_{S_{x_{n-1},x_{n+1}}}$. Furthermore, $Z_{t} = \widetilde Z_{t}$ on the set $\{t \geq H_\D\}$, while on the set $\{t < H_\D\}$ we see that $|\Im(Z_t)| < |\Im(\widetilde Z_t)|$. This implies that $\{|\Im(Z_{\tau_{S_{x_{n-1},x_{n+1}}}})| \geq \beta_n\} \subseteq \{|\Im(\widetilde Z_{\widetilde \tau_{S_{x_{n-1},x_{n+1}}}})| \geq \beta_n\}$, and the claim follows. Naturally, the case $y<0$ can be handled by a symmetric argument.

{\bf Claim 2:} For fixed $x_n \in {\cal X}$, 

$$
P_{x_n} (|I_{\tau_{S_{x_{n-1},x_{n+1}}}}| \geq \beta_n) \geq \frac{1}{2} P_{x_n}( |I_{\tau_{K^{\beta_n}_{x_{n-1},x_{n+1}}}}| = \beta_n ).
$$

To prove this, note that  

$$
P_{x_n} (|I_{\tau_{S_{x_{n-1},x_{n+1}}}}| \geq 1| I_{\tau_{K^{\beta_n}_{x_{n-1},x_{n+1}}}} = \beta_n) \geq P_{x_n} (I_{\tau_{S_{x_{n-1},x_{n+1}}}} \geq \beta_n| I_{\tau_{K^{\beta_n}_{x_{n-1},x_{n+1}}}} = \beta_n).
$$ 

This latter probability is precisely $\frac{1}{2}$ by the strong Markov property and symmetry. The symmetric argument shows that $P_{x_n} (|I_{\tau_{S_{x_{n-1},x_{n+1}}}}| \geq \beta_n| I_{\tau_{K^{\beta_n}_{x_{n-1},x_{n+1}}}} = -\beta_n) \geq \frac{1}{2}$ as well, and combining these yields the claim.

\begin{figure}[H]
\centering{}\includegraphics[width=9cm,height=9cm,keepaspectratio]{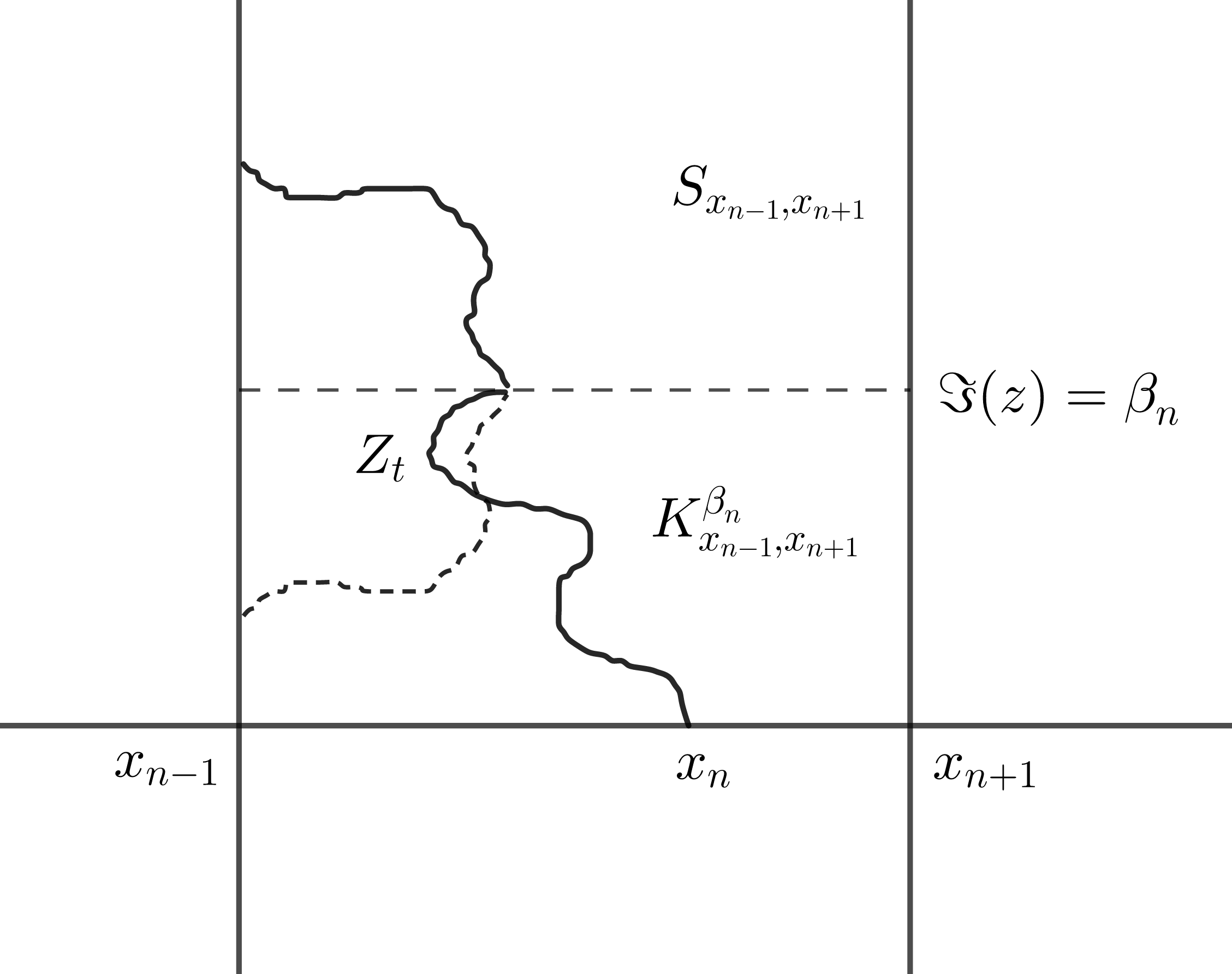}\caption{After hitting the top of $K^{\beta_n}_{x_{n-1},x_{n+1}}$, the Brownian motion is equally likely to exit $S_{x_{n-1},x_{n+1}}$ above $\{\Im(z)=1\}$ as below}
\end{figure}


Having established these claims, we can prove (\ref{xz}). 

\begin{equation}
\begin{split}
    \inf_{x_n \in {\cal X}, y \in (-b_n,b_n)} P_{x_n + yi} (|I_{\tau_{S_{x_{n-1},x_{n+1}}}}| \geq \beta_n) & \geq \inf_{x_n \in {\cal X}} \frac{1}{2} P_{x_n}( |I_{\tau_{K^{\beta_n}_{x_{n-1},x_{n+1}}}}| = \beta_n ) \\
     & = \inf_{n \in {\mathbb Z}} \frac{1}{2} P_{0}( |I_{\tau_{K^{\beta_n}_{-a_n,a_{n+1}}}}| = \beta_n ) \\
     & \geq \inf_{n \in {\mathbb Z}} \frac{1}{2} P_{0}( |I_{\tau_{K^{\beta_n}_{-\min(a_n,a_{n+1}),\min(a_n,a_{n+1})}}}| = \beta_n ) \\
    & \geq \frac{1}{2} P_{0}( |I_{\tau_{K^\ell_{-1,1}}}| = \ell ),
\end{split}
\end{equation}

since for a Brownian motion starting at $0$ the event $\{ |I_{\tau_{K^{\beta_n}_{-a_n,a_{n+1}}}}| = \beta_n \}$ is contained in the event $\{ |I_{\tau_{K^{\beta_n}_{-\min(a_n,a_{n+1}),\min(a_n,a_{n+1})}}}| = \beta_n \}$, and also $\frac{\beta_n}{\min(a_n,a_{n+1})} \leq \ell$.

\end{proof}

{\bf Remark:} The coupling argument used to prove Claim 1 is based on a method used in \cite[Thm. 1]{maxpmoment} in order to find the points which maximize the moments of the exit time from domains.

\vspace{12pt}

We may now complete the proof of Theorem \ref{bigguy}. By (\ref{yeungeun}) and the lines preceding it we have 

\[
E(\tau^{p})^{1/p}\leq \sum_{j=0}^{\infty} P(\tau_j < \tau)^{1/p} \max_{|n| \leq j}E_{0}(\tau_{S_{-a_n,a_{n+1}}}^p)^{1/p},
\]

Bounding these quantities by Lemmas \ref{pmoment_strip} and \ref{probexitstrip} yields

\[
E(\tau^{p})^{1/p}\leq \sum_{j=0}^{\infty} (\theta_0^{1/p})^j \max_{|n| \leq j+1}a_n^2.
\]


\section{Concluding remarks}

We have not attempted to optimize the conditions required in Theorem \ref{bigguy}, since it already covered quite general cases, including all domains with $b_n$ uniformly bounded and $a_n$ growing with at most a polynomial rate. Nevertheless, improvements using the same method are possible to suit particular situations if required. The following is an example; we will let $b_n=1$ for all $n$ in order to simplify the argument.

\begin{prop}
Suppose $(x_{n})_{n\in\mathbb{Z}}$ is an increasing sequence (with $x_0 = 0$) such that $\ell = \min_n (x_n - x_{n-1}) \geq 1$, $b_n=1$ for all $n$, and

\begin{equation} \label{keyeq2}
    \sum_{j=1}^\infty (\max_{|n| \leq j} a_n^2) (3/4)^{j/p} < \infty
\end{equation}

Then $E(\tau_{\mathscr{M}_{x}}^p)<\infty$.
\end{prop}

{\bf Proof:} This follows from the same method as was used to prove Theorem \ref{bigguy}, except that in this case since $\ell \leq 1$ we can put a simple upper bound on $\theta_0:= 1-\frac{1}{2} P_{0}( |I_{\tau_{K^\ell_{-1,1}}}| = \ell)$. Here $K^\ell_{-1,1}$ is contained in the square $K^1_{-1,1}$, and therefore $P_{0}( |I_{\tau_{K^\ell_{-1,1}}}| = \ell ) \geq P_{0}( |I_{\tau_{K^1_{-1,1}}}| = 1 )= \frac{1}{2}$, by symmetry. Thus, $\theta_0 \leq \frac{3}{4}$. The result follows from this. \qed

Thus, for instance, this proposition allows us to conclude that $E(\tau_{\mathscr{M}_{x}}^p)<\infty$ if $a_n = r^{|n|}$ with $r>1$, provided that $r< (\frac{4}{3})^{1/(2p)}$. No doubt this argument can be refined, if required.

There are a number of variants on the problem we have addressed, many of which can be handled by suitable adaptions of the method we have employed. We will describe one, again simplifying by setting $b_n =1$ for all $n$. Suppose we form a comb domain out of an increasing one-sided sequence $(x_{n})_{n=0}^\infty$ of real numbers without accumulation point in $\RR$ and with $x_0=0$. That is, we let $\mathscr{M}^+_{x}$ be the domain 
\[
\mathscr{M}^+_{x}:=\{\Re(z)>0\} \setminus\bigcup_{n=1}^\infty I_{n}
\]
where $I_{n}:=\{x_{n}\}\times(\{1,+\infty\}\cup\{-\infty,-1\})$. It may seem that we could weaken our conditions in order to conclude that $p$-th moments are finite, since this domain is essentially smaller than one would be corresponding to a two-sided sequence. However, the following proposition shows that this is not the case.

\begin{prop}
\begin{itemize}
    \item[a)] Suppose $\mathscr{M}_{x}$ is a comb domain corresponding to a two-sided sequence $(x_{n})_{n\in \mathbb{Z}}$. Then $E(\tau_{\mathscr{M}_{x}}^p)<\infty$ if, and only if, $E(\tau_{\mathscr{M}^+_{x}}^p)<\infty$ and $E(\tau_{\mathscr{M}^-_{x}}^p)<\infty$, where $\mathscr{M}^+_{x} = \mathscr{M}_{x} \cap \{\Re(z) > x_0\}$ and $\mathscr{M}^-_{x} = \mathscr{M}_{x} \cap \{\Re(z) < x_1\}$.
    
    \item[b)] Suppose $\mathscr{M}^+_{x}$ is a comb domain corresponding to a one-sided sequence $(x_{n})_{n=0}^\infty$, with $x_0=0$. Extend the sequence to a two-sided one by the rule $x_{-n} = -x_n$, and let $\mathscr{M}_{x}$ be the comb domain corresponding to this two-sided sequence (note that $\mathscr{M}^+_{x} = \mathscr{M}_{x} \cap \{\Re(z) > x_0\}$). Then $E(\tau_{\mathscr{M}^+_{x}}^p)<\infty$ if, and only if, $E(\tau_{\mathscr{M}_{x}}^p)<\infty$.
\end{itemize}
\end{prop}

{\bf Proof:} (sketch) It is clear that $(a)$ implies $(b)$, and the forward implication of $(a)$ is trivial since $\mathscr{M}^+_{x},\mathscr{M}^-_{x} \subseteq \mathscr{M}_{x}$. To prove the reverse implication, we apply the following result, which is Theorem 3 in \citep{markowsky}. 


\begin{thm} \label{flydag}
Suppose that $V$ and $W$ are domains with nonempty intersection, neither of which is contained in the other.  Suppose further that $E(T_V^p)< \infty$ and $E(T_W^p)< \infty$. Let $\tilde \dd V^+ = \dd V \cap W$ and $\tilde \dd W^+ = \dd W \cap V$, where $\dd V$ and $\dd W$ denote the boundaries of $V$ and $W$, and assume that the following conditions are satisfied:

\begin{itemize} \label{}


\item[(i)] $\sup_{a \in \tilde \dd V^+} E_a(T^p_W) < \infty$;

\item[(ii)] $\sup_{a \in \tilde \dd W^+} E_a(T^p_V) < \infty$;

\item[(iii)] $\sup_{a \in \tilde \dd V^+} P_a(B_{T_W} \in \tilde \dd W^+) < 1$.

\end{itemize}

Then $E(T_{V \cup W}^p)< \infty$.
\end{thm}

Here we take $V=\mathscr{M}^+_{x}$ and $W=\mathscr{M}^-_{x}$, and thus $\tilde \dd V$ and $\tilde \dd W$ are the line segments between $x_0 \pm i$ and between $x_1 \pm i$, respectively. It can then be shown using methods similar to those employed in the proof of Theorem \ref{bigguy} above that $(i)-(iii)$ hold; in particular, the construction used to prove Claim 1 above can be adapted to show that the suprema in $(i)-(iii)$ are all attained at points with imaginary part $0$. Details are omitted.

An anonymous referee has asked the following question:

\vspace{12pt}

{\bf Question:} Given $p<q$, can we construct a comb domain $\mathscr{M}_{x}$ with finite $p$-th moment but infinite $q$-th moment?

\vspace{12pt}

Unfortunately, our methods do not seem to be able to give lower bounds on the moments, so we do not know how to construct such a domain. We think it is a nice open problem, though, and have included it for this reason.
\section{Acknowledgements}

The authors would like to thank an anonymous referee for valuable comments, including a suggestion which led to a significant generalization in our results.

\bibliographystyle{plain}
\bibliography{Maheref}

\end{document}